\documentclass{article}
\usepackage{amsfonts,amsmath,amssymb,amsthm}

\newtheorem{theorem}{Theorem}
\newtheorem{lemma}{Lemma}[section]
\newtheorem{corollary}[lemma]{Corollary}
\newtheorem{proposition}[lemma]{Proposition}
\newtheorem{remark}[lemma]{Remark}
\newtheorem{definition}[lemma]{Definition}

\def\R{\mathbb{R}}
\def\E{\mathcal E}

\def\H{\mathcal{H}}

\def\Z{\mathbb{Z}}

\def\H{\mathcal{H}}
\def\eps{\varepsilon}
\newcommand{\tr}{\operatorname{\text{tr}}}
\begin{document}

\title{Curvature flow in heterogeneous media}

\author{
Annalisa Cesaroni\footnote{
Dipartimento di Matematica Pura e Applicata,
Universit\`a di Padova,
via Trieste 63, 35121 Padova, Italy}  
\and Matteo Novaga\footnotemark[\value{footnote}]
\and Enrico Valdinoci\footnote{
Dipartimento di Matematica,
Universit\`a di Roma Tor Vergata,
Via della Ricerca Scientifica 1,
00133 Roma, Italy}}

\date{}

\maketitle

\begin{abstract}
\noindent 
In recent years, there has been a growing interest in geometric evolution in heterogeneous media. Here we consider curvature driven flows of planar curves, with an additional space-dependent forcing term. Motivated by a homogenization problem, we look for estimates which depend only on the $L^\infty$-norm of the forcing term. By means of an asymptotic analysis, we discuss the properties of the limit solutions of the homogenization problem, which we can rigorously solve in some special cases:
that is, when the initial curve is a graph, and the forcing term does not depend on the vertical direction. As a by-product, in such cases we are able to define a soluton of the geometric evolution when the forcing term is just a bounded, not 
necessarily continuous,  function. 
\end{abstract}

\section{Introduction}\label{secintro}

In this paper we consider the  curvature shortening flow of planar curves in a heterogeneous medium,
which is modeled by a spatially-dependent additive forcing term. The evolution law reads:
\begin{equation}\label{eqkg}
v = (\kappa + g)\nu,
\end{equation}
where $\nu$ is inward normal vector to the curve, $\kappa$ is the curvature of the curve,
$v$ is the normal velocity vector, and $g\in L^\infty(\R^2)$ represents the forcing term.

The original motivation for our analysis comes from a homogenization problem related to   
the averaged behaviour of an interface moving by curvature plus 
a rapidly oscillating forcing term. More precisely, the evolution law is given by 
\begin{equation}\label{eqkgomo}
v=\left(\kappa+   g\left(\frac{x}{\eps},\frac{y}{\eps} \right)\right)\nu,
\end{equation}
where  $g$ is a $1$-periodic Lipschitz continuous function.
 
When the forcing term is periodic, equation \eqref{eqkg} 
was recently considered in \cite{DKH:08}, where the authors 
prove existence and uniqueness of planar pulsating waves in every direction of propagation.
This result leads to the homogenization of \eqref{eqkgomo} for plane-like initial data (see Section \ref{secomo}). 
Related results on the homogenization of interfaces moving with normal velocity given by    
\[
v=\left(\eps \kappa+   g\left(\frac{x}{\eps},\frac{y}{\eps}\right)\right)\nu,
\] 
have been obtained  in \cite{CB:04} and \cite{LS:05}, 
under suitable assumptions on the forcing term 
including the fact that it does not change sign, 
and in \cite{CLS:09} under more general assumptions. 
In particular,  the authors show that the homogenized evolution law, when it exists,  
is  a  first order anisotropic geometric law of the form $v=\overline{c}(\nu)\,\nu$.  

Coming back to our problem, as a  first step we look for geometric estimates 
for solutions to \eqref{eqkg}, which depend 
only on the $L^\infty$-norm of $g$. In particular,
reasoning as in the case of the unperturbed curvature flow \cite{Hu:90, A:91},
in Section \ref{secex} we classify all possible singularities
which can arise during the evolution. As a consequence, in Section \ref{seciso}
we can show that, when $g$ is smooth and the initial curve is embedded, the existence
time of a regular solution to \eqref{eqkg} is bounded below by a quantity
depending only on $\| g\|_\infty$ and on the initial curve.
Unfortunately, since we have no estimates on the curvature in terms of $\| g\|_\infty$, 
we are not able to obtain a general existence result for \eqref{eqkg}
in the nonsmooth case, i.e. when $g\in L^\infty$.

However, in Section \ref{secgr} we overcome this difficulty by assuming that the initial
curve is the graph of a function $u$, for instance in the vertical direction.
In this case equation \eqref{eqkg} becomes
\begin{equation}\label{eqkggraph}
u_t = \frac{u_{xx}}{1+u_{x}^2} + g(x,u(x)) \sqrt{1+u_{x}^2}\,.
\end{equation}
In Lemma \ref{lemest} we establish an $L^p$-estimate
on $u_x$, which depend only on $\| g\|_\infty$.
In Proposition \ref{provo} we consider a sequence of smooth forcings
$g_n$ weakly converging to $g\in L^\infty$.
Using the estimate on $u_x$ and the results of the Section \ref{seciso}, and letting $u_n$ be
the solution corresponding to $g_n$, we can pass to the limit as $n\to\infty$ and obtain
that $u_n\to u\in H^1([0,T],L^2([0,1]))\cap L^\infty([0,T],H^1([0,1]))$, for some time $T>0$ 
depending only on $\| g\|_\infty$ and on the initial datum. 
When $g$ does not depend on $u$, we obtain a stronger estimate on $\| u_t\|_\infty$, 
which allows us to show that $u\in W^{1,\infty}([0,T],L^\infty([0,1]))\cap L^\infty([0,T],W^{2,\infty}([0,1]))$.

As a first application, this leads to an existence and uniqueness result for solutions to \eqref{eqkggraph},
when $g$ is a $L^\infty$-function which is independent of $u$ (see Theorems \ref{teogi} and \ref{teocfr}).
 

The second application of our result is to the homogenization problem \eqref{eqkgomo}. 
In section \ref{secomo},
under the assumptions of Theorem \ref{teogi}, that is, when the curve is a graph
and $g$ is independent of the vertical direction, we can pass to the limit in \eqref{eqkgomo}
as $\eps\to 0$, and show that the limit curve moves according to the evolution law  
\begin{equation}\label{eqkgeff}
v = \left(\kappa + \int_{[0,1]^2} g( x,y)d  x dy  \right)\nu.
\end{equation}

In Section \ref{secomo2}, by means of a formal asymptotic analysis, we discuss the limit behavior 
of \eqref{eqkgomo} in the general case.
In particular, we show that the solutions are expected to converge, in the viscosity sense, to a solution of
the geometric equation
\begin{equation}\label{eqkglim}
v = \left(\kappa + \overline{c}(\nu)\right)\nu,
\end{equation}
where the function $\overline{c}\in L^\infty(S^1)$ is generally {\em discontinuous}. The main obstacle to a rigorous
analysis of \eqref{eqkgomo}, for instance  using the level set method along the lines of \cite{e}, \cite{LS:05}, 
is due to the fact that a viscosity theory for \eqref{eqkglim} is presently not available.

\paragraph{Acknowledgements.}
The second author wish to thank the University of Tours
and the Research Institute {\em le Studium} for the kind hospitality and support.

\section{Local existence of solutions}\label{secex}

In this section we are concerned with the local existence 
for \eqref{eqkg}, under the assumption that the forcing term  
$g$ is smooth and bounded, i.e. $g\in C^{\infty}(\R^2)\cap L^\infty(\R^2)$.
If we parametrize counterclockwise the evolving curve with a function $\gamma:[0,1]\times [0,T]\to \R^2$,
$\gamma=(\gamma^1,\gamma^2)$, problem \eqref{eqkg} becomes
\begin{equation}\label{eqpar}
\gamma_t = (\kappa+g)\nu
= \frac{\gamma_{xx}^\perp}{|\gamma_x|^2} + g(\gamma)\frac{(-\gamma^2_x,\gamma_x^1)}{|\gamma_x|},
\end{equation}
where $\xi^\perp$ denotes the component of the vector $\xi$ orthogonal to $\gamma$.
As usual we let $\tau, \nu, \kappa$ be respectively the unit tangent vector, the unit normal vector  
and the curvature of the evolving curve.
Denoting by $s$ the arclength paramter of the curve, so that $\partial_s=\partial_x/|\gamma_x|$,
by the classical Frenet--Serret formulas we have 
\begin{equation}\label{6bis}
\gamma_s = \tau, \qquad \gamma_{ss} = \tau_s = \kappa\,\nu, \qquad \nu_s = -\kappa\,\tau.
\end{equation}
Following \cite{HP:99}, we give a local in time existence result for \eqref{eqpar}.

\begin{theorem}\label{smallex}
Let $\gamma_0:[0,1]\to \R^2$ be a smooth map such that $|\gamma_0'(x)|>0$ for all $x\in [0,1]$,
then there exist $T>0$ and  a smooth solution to \eqref{eqpar}, 
defined on $[0,1]\times [0,T]$, such that $\gamma(x,0)=\gamma_0(x)$ for all $x\in [0,1]$.
\end{theorem}

\begin{proof}
The proof is standard and we only sketch it. If we write $\gamma([0,1],t)$ 
as graph of a function $f(x,t)$ over the initial curve $\gamma_0([0,1])$, so that
\[
\gamma(x,t)=\gamma_0(x)+f(x,t)\nu_0(x),
\]
equation \eqref{eqpar} becomes
\begin{eqnarray}\label{eqgraf}
f_t &=& \frac{\gamma_{xx}\cdot \nu_0}{|\gamma_x|^2}-g(\gamma) (\nu\cdot\nu_0)
\\ \nonumber
&=& \frac{f_{xx}+\kappa_0(1- \kappa_0f)|\gamma'_0|^2}{|f_x|^2+(1-\kappa_0 f)^2|\gamma'_0|^2} -
\frac{(1-\kappa_0 f)|\gamma'_0|\,g(\gamma_0(x)+f(x,t)\nu_0(x)) }{\sqrt{|f_x|^2+(1-\kappa_0 f)^2|\gamma'_0|^2}}.
\end{eqnarray}
Since \eqref{eqgraf} is a uniformly parabolic quasilinear equation, 
the thesis follows by standard semigroup techniques, see 
for instance \cite{LSU,Lu:95}.
\end{proof}
%
\subsection{Estimates on the curvature and its derivatives} 

\begin{lemma}
The following commutation rule holds:
\begin{equation}\label{eqcomm}
\partial_t \partial_s = \partial_s\partial_t +\kappa(\kappa+g)\partial_s.
\end{equation}
Moreover,
\begin{eqnarray}
\label{eqkg5}
\tau_t &=& (\kappa + g)_s\nu  ,
\\ \label{eqnu}
\nu_t &=& - (\kappa+g)_s\tau 
\\ \label{eqkappa}
\kappa_t &=& (\kappa+g)_{ss} + \kappa^2(\kappa+g).
\end{eqnarray}\end{lemma}

\begin{proof} By definition of
arclength, we have $$ \partial_s = \frac{\partial_x }{|\gamma_x|}.$$
Therefore, from \eqref{eqpar} and \eqref{6bis},
\begin{eqnarray*}
\partial_t\partial_s-\partial_s\partial_t &=&
-|\gamma_x|^{-3} \gamma_x \cdot\gamma_{xt}\partial_x
\\
&=& -|\gamma_x|^{-2} \tau \cdot\gamma_{xt}\partial_x
\\
&=&
-|\gamma_x|\tau\cdot \big((\kappa+g)\nu\big)_s\partial_x
\\ 
&=&
-|\gamma_x|\tau\cdot (\kappa+g)\nu_s
\partial_x
\\
&=& 
\kappa (\kappa+g)\partial_s, 
\end{eqnarray*} that is \eqref{eqcomm}.
Now, applying \eqref{eqcomm} to \eqref{eqpar}
and \eqref{6bis}, we obtain
\begin{eqnarray*}
\tau_t &=& (\gamma_s)_t= (\gamma_t)_s+\kappa(\kappa+g)\gamma_s\\
&=& (\kappa+g)_s \nu+(\kappa+g)\nu_s+\kappa(\kappa+g)\tau=
(\kappa+g)_s\nu,
\end{eqnarray*}
which is \eqref{eqkg5}.

Also, since $|\nu|=1$,
$$ 0=\frac{(\nu\cdot \nu)_t}{2}= \nu\cdot \nu_t$$
and so, from \eqref{eqkg5},
\begin{eqnarray*}
\nu_t &=&(\nu_t\cdot \nu)\nu+(\nu_t\cdot \tau)\tau=
(\nu_t\cdot \tau)\tau\\
&=& \big(
(\nu\cdot\tau)_t-\nu\cdot\tau_t
\big)\tau=-(\nu\cdot\tau_t)\tau=-(\kappa+g)_s\tau,
\end{eqnarray*}
that is \eqref{eqnu}, and
\begin{eqnarray*}
\kappa_t &=& (\kappa\nu)_t\cdot \nu=
(\tau_s)_t\cdot\nu=(\tau_t)_s\cdot\nu+
\kappa(\kappa+g)\tau_s\cdot\nu
\\ 
&=&
\big( (\kappa+g)_s\nu\big)_s\cdot\nu+\kappa^2(\kappa+g)
=(\kappa+g)_{ss}+\kappa^2(\kappa+g)
,\end{eqnarray*} 
that is \eqref{eqkappa}.\end{proof}

Let us compute the evolution for the spatial derivaties of the curvature.
We denote by $p_{j,k}(\partial_s^\ell\kappa,\partial^m_s g)$ 
a generic polynomial depending on the derivatives 
up to order $j$ of $\kappa$ and the derivatives up to order $k$ of $g$.  

\begin{lemma}
For all $j\in\mathbb N$, $j\ge1$, it holds
\begin{equation}\label{eqkj}
\partial_t \partial_s^j\kappa = (\partial_s^j\kappa)_{ss} 
+ \left( (j+3)\kappa^2 + (j+2)\kappa g\right)\partial_s^j\kappa 
+p_{j-1,j+2}(\partial_s^\ell\kappa,\partial^m_s g).
\end{equation}
\end{lemma}

\begin{proof}
The proof is by induction on $j$. When $j=1$ from 
\eqref{eqcomm} and \eqref{eqkappa} we easily get
\[
\partial_t \kappa_s = (\kappa_s)_{ss} + (4\kappa^2+3\kappa g)\kappa_s + (\kappa^2g_s+g_{sss}).
\]
Assume now \eqref{eqkj} for some $j\in\mathbb N$. Using \eqref{eqcomm}, we compute recursively
\begin{eqnarray*}
\partial_t \partial_s^{j+1}\kappa &=& \partial_s \partial_t \partial_s^j\kappa + \kappa(\kappa+g)\partial_s^{j+1}\kappa
\\
&=& (\partial_s^{j+1}\kappa)_{ss} + 
\left((j+3)\kappa^2 + (j+2)\kappa g + \kappa(\kappa+g)\right)\partial_s^{j+1}\kappa 
+ p_{j,j+3}(\partial_s^\ell\kappa,\partial^m_s g),
\end{eqnarray*} 
which gives \eqref{eqkj} for all $j$.
\end{proof}

We now compute the evolution equation of $w:=\log |\gamma_x|$.

\begin{lemma}
There holds
\begin{equation}\label{eqw}
w_t = -\kappa(\kappa+g).
\end{equation}
\end{lemma}

\begin{proof}
A direct computation using \eqref{eqpar} gives
\begin{equation*}
w_t = \frac{\gamma_x\cdot\gamma_{xt}}{|\gamma_x|^2} = \tau\cdot (\partial_s\gamma_t) = -\kappa(\kappa+g).
\qedhere\end{equation*}
\end{proof}

\begin{lemma}\label{stimecurvature}
Assume that \eqref{eqpar} admits a smooth solution on $[0,\bar t]$, with $\bar t>0$. Then
\[
\max_{[0,1]\times [0,\bar t]}(\partial_s^j \kappa)^2 \le C_j
\]
for all $j\in\mathbb N$, where the constants $C_j$ depend only on $\bar t$,  
on $\max_{[0,1]\times [0,\bar t]}\kappa^2$ and $\|g\|_{C^{j+2}}$.
\end{lemma}

\begin{proof}
Following~\cite{GH}, we let
\[
K_j(x,t):= (\partial^j_s \kappa(x,t))^2 \qquad \qquad  
M_j(t):=\max_{x\in [0,1]}K_j(x,t).
\]
For all $\bar x$ such that $K_j(\bar x,t)=M_j(t)$
we have 
\begin{eqnarray*}
\partial_s K_j &=& 0
\\
\partial_{ss} K_j &=&  2\Big( (\partial^j_s \kappa)_s^2 + \partial^j_s \kappa
(\partial^j_s \kappa)_{ss} \Big) \le 0
\end{eqnarray*}
and so
\begin{equation}\label{HT1}
\partial^j_s \kappa(\bar x,t)
(\partial^j_s \kappa)_{ss}(\bar x,t) \le 0.
\end{equation}
Recalling \eqref{eqkj}, for a.e. $t\in [0,\bar t]$ we have
\begin{eqnarray*}
\dot M_j(t) &=& \max_{\bar x:\ K_j(\bar x,t)=M_j(t)}
\partial_t K_j
\\ 
&=& \max_{\bar x:\ K_j(\bar x,t)=M_j(t)}
2\partial^j_s \kappa (\partial^j_s \kappa)_t
\\
&=& \max_{\bar x:\ K_j(\bar x,t)=M_j(t)} 2\partial^j_s \kappa
\left((\partial_s^j\kappa)_{ss} 
+ A_j\partial_s^j\kappa +B_j\right)
\end{eqnarray*}
where the constants $A_j$, $B_j$ depend on $M_\ell$ and $\|g\|_{C^k}$, with $\ell<j$ and $k\le j+2$.
Hence, using \eqref{HT1} we get
\[
\dot M_j \le 2 A_j M_j + 2 B_j. 
\]
By Gronwall's Lemma it then follows that the quantities $M_j$ are uniformly bounded on $[0,\bar t]$.
\end{proof}

Since the existence result in Theorem \ref{smallex} is first established in the usual H\"older parabolic spaces 
$C^{k+\alpha,2(k+\alpha)}([0,1]\times [0,T])$ (see \cite{Lu:95}), 
if we still denote by $T$ the maximal existence time of the evolution, we have that,
if $T<+\infty$, 
either $|\gamma_x|^{-1}$ or $|\partial^j_s\kappa|$ blow up as $t\to T$, for some $j\in\mathbb N$.

\begin{proposition}\label{propex}
Let $T$ be the maximal existence time of the evolution \eqref{eqpar},
and assume $T<+\infty$. Then
\begin{equation}\label{eqblow}
\lim_{t\to T} \|\kappa^2\|_{L^\infty} = +\infty.
\end{equation}
\end{proposition}

\begin{proof}
Assume by contradiction that $\kappa^2$ is uniformly bounded for all $t\in [0,T)$ and $x\in [0,1]$. Equation
\eqref{eqw} implies that $|\gamma_x|$ and $1/|\gamma_x|$ are also uniformly bounded on $[0,1]\times[0,T)$.
But Lemma \ref{stimecurvature} implies that also the quantities $(\partial_j \kappa(x,t))^2$
are uniformly bounded on $[0,1]\times[0,T)$ for all $j\in \mathbb N$,
thus reaching a contradiction. We then proved  
\begin{equation*}
\limsup_{t\to T} \|\kappa^2\|_{L^\infty} = +\infty.
\end{equation*}
Notice that the $\limsup$ is indeed a full limit due to \eqref{eqkappa}.
\end{proof}

The following Lemma provides a lower bound to \eqref{eqblow}.

\begin{lemma}
Let $T$ as above and assume $T<+\infty$. The following curvature lower bound holds:
\begin{equation}\label{stimak}
\liminf_{t\to T}\sqrt{T-t}\,\|\kappa\|_{L^\infty}  \ge \frac{1}{\sqrt{2}}.
\end{equation}
\end{lemma}

\begin{proof}
Notice that \eqref{eqkappa} can be written as 
\begin{equation}\label{eqkbis}
(\kappa+g)_t = (\kappa +g)_{ss} + (\kappa+g)\kappa^2 + (\kappa+g)\nabla g\cdot\nu.
\end{equation}
Letting $w:=(\kappa +g)^2$ and $\eps>0$, from \eqref{eqkbis} it follows
\begin{eqnarray}\label{eqktris}
\nonumber w_t &=& w_{ss} -2 (\kappa+g)_s^2 + 2 w(\sqrt w-g)^2 + 2w\nabla g\cdot\nu
\\
&\le& w_{ss} + 2 w\left(w-2g\sqrt w +g^2\right) + 2\|\nabla g\|_{L^\infty} w
\\
\nonumber &\le& w_{ss} + 2 w\left( (1+\eps)w+\left( 1+\frac{1}{\eps}\right)g^2\right) + 2\|\nabla g\|_{L^\infty} w
\\
\nonumber &\le& w_{ss} + 2(1+\eps) w^2 + 2\left(2 + \frac 1 \eps\right)\|g\|_{W^{1,\infty}}w.
\end{eqnarray}
Letting $M:=\max_{x\in [0,1]}(\kappa +g)^2$, from \eqref{eqktris} we get
\begin{equation}\label{eqM}
\frac{\rm d}{{\rm d}t}{(M+C)} = \dot M \le 2(1+\eps) M^2 + 2\left(2 + 
\frac 1 \eps\right)\|g\|_{W^{1,\infty}}M\le 2(1+\eps)(M+ C)^2,
\end{equation}
where $C=[(1+1/(2\eps))/(1+\eps)]\,\|g\|^2_{W^{1,\infty}}$, so that 
\[
-\frac{\rm d}{{\rm d} t} \frac{1}{M+C}\le 2(1+\eps).
\]
Integrating on $[t,s]\subset [0,T)$ we thus obtain 
\[
\frac{1}{M(t)+C}-\frac{1}{M(s)+C}\le 2(1+\eps)(s-t).
\]
Letting now $s\to T$ and recalling that $M(s)\to +\infty$ by Proposition \ref{propex}, we get
\[
\frac{1}{M(t)+C}\le 2(1+\eps)(T-t),
\]
that is
\begin{equation}\label{eqMbis}
M(t)\ge \frac{1}{2(1+\eps)(T-t)} - C,
\end{equation}
which gives the thesis.
\end{proof}

{}From  \eqref{eqMbis} and Proposition \ref{propex}  we obtain the following
estimate on the maximal existence time of the evolution.

\begin{proposition}\label{protime}
Let $T$ be the maximal existence time of \eqref{eqpar}, then 
\[
T \ge c(\|\kappa_0\|_{L^\infty}, \|g\|_{W^{1,\infty}}).
\] 
\end{proposition}

Notice that if the initial curve is embedded then, thanks to Proposition \eqref{protime},
it remains embedded in a time interval $[0,T']$, with $T'>0$ depending only on the initial datum 
and on $\|g\|_{W^{1,\infty}}$.

We think it is an interesting problem to determine whether or
not the constant $c$ in Proposition \ref{protime} depends only on the 
initial set and on the $L^\infty$-norm of $g$
(see for instance Section \ref{seciso} below for a special case).

\subsection{Huisken's monotonicity formula} 

In the following we derive a monotonicity formula for curvature flow with a forcing term, 
and apply it to the analysis of singularities. \\
By a standard computation, using the fact that $\gamma$ solves \eqref{eqpar}, we get the following formula.

\begin{lemma}\label{lemma1}
Let $\tau>0$ and let $f:\R^2\times[0,\tau)\to \R$ be a smooth function. Then 
\begin{equation}\label{eqf}
\frac{d}{dt}\int_\gamma f(\gamma(x(s),t) ,t)ds
=\int_\gamma \left[ f_t -\kappa(\kappa+g)f  +(\kappa+g)\nabla f \cdot\nu \right]ds.
\end{equation}
\end{lemma}
\noindent We denote by
$L_t(\gamma)$ the length of the curve $\gamma([0,1],t)$, that is
$$ L_t(\gamma):=\int_0^1 |\gamma_x|dx= \int_\gamma ds.$$
When no confusion can arise, we write $L(\gamma)$ instead of $L_t(\gamma)$.

\begin{corollary}\label{corper} 
Let $\gamma:[0,1]\times [0,T]\to \R^2$ be a solution to \eqref{eqpar}.
The following estimates hold:  
\begin{eqnarray}\label{paris}
L_t(\gamma ) &\le& L_0(\gamma )\,e^{\frac{\| g\|_\infty^2}{2}t}
\qquad \qquad\forall t\in [0,T] 
\\ \label{tours}
\int_0^T \int_\gamma\kappa^2 dsdt &\le &2\left( L_t(\gamma )-L_0(\gamma )\right) + \| g\|_\infty^2\,T
\\
&\le& 2L_0(\gamma)\left(e^{\frac{\| g\|_\infty^2}{2}t}-1\right) + \| g\|_\infty^2\,T.    
\end{eqnarray} 
\end{corollary}

\begin{proof} Taking $f\equiv 1$ in \eqref{eqf} we have
\begin{equation}\label{eqL}
\frac{\partial}{\partial t}L_t(\gamma)= -\int_\gamma \kappa(\kappa+g)ds
\le \int_\gamma -\frac{\kappa^2}{2} + \frac{g^2}{2}ds, 
\end{equation} 
which gives \eqref{paris} by Gronwall's Lemma. 
Estimate \eqref{tours} also follows by integrating \eqref{eqL} on $[0,T]$.
\end{proof} 
 
We now apply Lemma \ref{lemma1} with $f(p,t)= \frac{e^{-|p-p_0|^2/4(T-t)}}{\sqrt{4\pi(T-t)}}$, $p,p_0\in \R^2$, 
and we get 
\begin{eqnarray}\label{conto1}
\nonumber && \frac{d}{dt}\int_\gamma \frac{e^{-|\gamma(x(s),t) -p_0|^2/4(T-t)}}{\sqrt{4\pi(T-t)}}ds
\\ \nonumber
&& = -\int_\gamma \frac{e^{-|\gamma-p_0|^2/4(T-t)}}{\sqrt{4\pi(T-t)}}\left[ 
 \frac{|\gamma -p_0|^2}{4(T-t)^2}-\frac{1}{2(T-t)} +\kappa(\kappa+g)  
 +(\kappa+g)\frac{(\gamma -p_0)\cdot \nu}{2(T-t)}   \right]ds
\\\nonumber
&& =-\int_\gamma \frac{e^{-|\gamma -p_0|^{2}/4(T-t)}}{\sqrt{4\pi (T-t)}}
\left[ \frac{\gamma -p_0}{2(T-t)} + \left(\kappa +\frac{g}{2}\right)\nu \right]^{2} ds
+\frac{1}{4}\int_\gamma \frac{e^{-|\gamma -p_0|^{2}/4(T-t)}}{\sqrt{4\pi (T-t)}}  g^{2} ds
\\
&& \quad +\int_\gamma  
\frac{e^{-|\gamma -p_0|^{2}/4(T-t)}}{\sqrt{4\pi(T-t)}}\left[\frac{1}{2(T-t)}  
+\kappa \frac{(\gamma -p_0)\cdot \nu}{2(T-t)} \right]ds.
\end{eqnarray} 
Following \cite[Theorem 3.1]{Hu:90}, the last term can be actually written as
\begin{eqnarray*}
&& \int_\gamma  \frac{e^{-|\gamma-p_0|^2/4(T-t)}}{\sqrt{4\pi(T-t)}}\left[\frac{1}{2(T-t)}   
+\kappa
\frac{(\gamma -p_0)\cdot \nu}{2(T-t)} 
\right]ds
\\ && =\int_\gamma  \frac{e^{-|\gamma-p_0|^2/4(T-t)}}{\sqrt{4\pi(T-t)}}\left|
\frac{(\gamma -p_0)\cdot \tau}{2(T-t)}
\right|^2 ds.
\end{eqnarray*}
Substituing this in \eqref{conto1} we obtain an analog of Huisken's monotonicity formula (\cite{Hu:90}) 
\begin{eqnarray}\label{monotonicity}
&& \frac{d}{dt}\int_\gamma \frac{e^{-|\gamma(x(s),t)-p_0|^2/4(T-t)}}{\sqrt{4\pi(T-t)}}ds
\\\nonumber 
&& = \int_\gamma  \frac{e^{-|\gamma(x(s),t)-p_0|^2/4(T-t)}}{\sqrt{4\pi(T-t)}}\left(-\left[ \kappa+
\frac{(\gamma(x(s),t)-p_0)\cdot \nu}{2(T-t)} 
+\frac{g }{2} \right]^2  +\frac{1}{4}  g^2  \right)ds.
\end{eqnarray}
In the next paragraph we will apply this formula in the analysis of type I singularities. 


\subsection{Type I singularities} 
We assume that at time $T$ the flow is developing a singularity of type I, i.e. there exists a constant $C_0>1$ such that 
\begin{equation}\label{sing1}
\max_{x\in[0,1]} |\kappa(x,t)|\leq \frac{C_0}{\sqrt{2(T-t)}}.
\end{equation} 

Observe that for every $x$ and $0\leq t\leq r<T$  
\begin{equation}\label{gamma1}
|\gamma(x,r)-\gamma(x,t)|\leq C_0 \sqrt{2(T-t)}-C_0 \sqrt{2(T-r)} +\|g\|_\infty (r-t).
\end{equation}
This implies that the functions $\gamma(\cdot, t)$ converge uniformly 
to a function $\gamma_T$ as $t\to T$. 
Now we fix   $x\in [0,1]$ such that $\gamma(x,t)\to \gamma_T(x)=:\widehat{p}$ and  $\kappa(x,t)$ becomes unbounded as $t\to T$. 
We rescale the curve around the point $\widehat{p} $ as follows: 
\[
\tilde{\gamma}(x,z):=\frac{\gamma(x,t(z))-\widehat{p}}{\sqrt{2(T-t(z))}}
\qquad z(t):= -\log\sqrt{T-t}.
\]
{}From \eqref{gamma1}, we deduce \[|\tilde{\gamma}(x,z)|=\left|\frac{\gamma(x,t(z))-\widehat{p}}{\sqrt{2(T-t(z))}}\right|\leq C_0 + \frac{1}{\sqrt{2}}\|g\|_\infty e^{-z}\leq C_0 + \frac{1}{\sqrt{2}}\|g\|_\infty ,\] 
so in particular  $\tilde{\gamma}(x,z)$ remains bounded as $z\to +\infty$. The evolution law satisfied by  the rescaled curve 
$\tilde\gamma$ is 
\begin{equation}\label{eqkgrescaled}
\tilde{\gamma}_z=(\tilde{\kappa}+\sqrt{2}e^{-z}g)\tilde{\nu}+\tilde{\gamma}
\qquad z\in [-\log\sqrt{T}, +\infty).
\end{equation}
We also the rescaled version of the monotonicity formula \eqref{monotonicity}: letting $y = \frac{x-\widehat{p}}{\sqrt{2(T-t)}}$,
we compute
\begin{eqnarray}\label{monorisc}
\nonumber 
&&\frac{d}{dz} \int_{\tilde{\gamma}}  e^{-|\tilde{\gamma}(x(\tilde{s}),z)|^2/2}  d\tilde{s}
 =  2(T-t)\frac{d}{dt}\int_{\gamma} \frac{e^{-|\gamma(x(s),t)-\widehat{p}|^2/4(T-t)}}{\sqrt{2(T-t)}}\,ds
\\
&& =   \int_{\tilde{\gamma}}  e^{-|\tilde{\gamma}(x(\tilde{s}),z)|^2/2} \left(-
\left[  \tilde{\kappa}+ \left\langle  \tilde{\gamma}(x(\tilde{s}),z)|\tilde{\nu}
\right\rangle +e^{-z}\frac{g}{\sqrt{2}} 
\right]^2 +\frac{e^{-2z}}{2}  g^2 \right)  d\tilde{s} .
\end{eqnarray}

Letting $F(z):=\int_{\tilde{\gamma}}   e^{-|\tilde{\gamma}(x(\tilde{s}),z)|^2/2}  d\tilde{s}$, 
equation \eqref{monorisc} gives
\begin{equation}\label{eqfz}
\frac{d}{dz}F(z) \leq  \frac{\|g\|^2_\infty}{2} e^{-2z}F(z)\le \frac{\|g\|^2_\infty}{2} e^{-2z}.
\end{equation}
Integrating \eqref{eqfz} we obtain 
\[
F(z)\le e^{\|g\|^2_\infty T/4}F(-\log\sqrt{T})
\qquad \forall z\ge -\log\sqrt{T}.
\]
In particular, we deduce that for every $R>0$ there exists a uniform bound on $\H^1(\tilde{\gamma}([0,1], z)\cap B(0,R))$. Indeed 
\begin{eqnarray}\label{stima}
\H^1(\tilde{\gamma}([0,1], z)\cap  B(0,R))&=& \int_{\tilde{\gamma}} \chi_{ B(0,R)} d\tilde{s} \leq \int_{\tilde{\gamma}} \chi_{ B(0,R)}  e^{\frac{R^2-|\tilde{\gamma}(x,z)|^2}{2}}d\tilde{s} 
\\ \nonumber 
&\leq&   e^{\frac{|R|^2}{2}} F(z)\leq K, 
\end{eqnarray} 
for some positive constant $K$.

\begin{proposition}\label{sing1prop} 
Under assumption \eqref{sing1}, for each sequence $z_j\to +\infty$ there exists a subsequence $z_{j_k}$ 
such that the curve $\tilde{\gamma}(\cdot, z_{j_k})$, rescaled around $\widehat{p}$, locally smoothly converges to some smooth, 
nonflat limit curve $\tilde{\gamma}_\infty$, such that 
\begin{equation}\label{eqomo}
\tilde{\kappa} +\left\langle \tilde{\gamma}_\infty|\tilde{\nu} \right\rangle =0.
\end{equation}
\end{proposition} 

\begin{proof}
The proof follows the same argument as in \cite[Proposition 3.4]{Hu:90}. Indeed, 
the limit curve is smooth thanks to \eqref{stima},
Proposition \ref{propex} and the fact that the rescaled curve $\tilde\gamma$ 
has uniformly bounded curvature. Moreover, it is nonflat by \eqref{stimak}.
Finally, the limit curve satisfies \eqref{eqomo} thanks to \eqref{monorisc} and \eqref{eqfz}.
\end{proof} 

\begin{remark}\label{remuno}\rm
Proposition \ref{sing1prop} implies that the type I singularities of \eqref{eqkg}
are modeled by homothetic solutions of the flow. 
We recall all such solutions correspond to closed curves and 
that the circle is the only embedded one (see \cite{AL:86}), hence  
$T$, under assumption \eqref{sing1}, is actually the extinction time for the evolution. 
{}From this we can conclude that 
\[
T \ge c(\|\kappa_0\|_{L^\infty}, \|g\|_{L^\infty}).
\] 
\end{remark}

\subsection{Type II singularities} 

We consider now the case   that at time $T$ the flow is developing a singularity of type II, i.e. 
\begin{equation}\label{sing2}\limsup_{t\to T}\max_{x\in[0,1]} |\kappa(x,t)| \sqrt{ T-t }=+\infty.\end{equation} 
\begin{proposition}
Under condition \eqref{sing2}, there exists a sequence of points and times $(x_n,t_n)$  on which
the curvature blows up such that the rescaled curve along this sequence converges in $\mathcal{C}^{\infty}$  to a
planar, convex limiting solution, which moves by translation.
\end{proposition} 
\begin{proof}
By means of \eqref{eqkappa}, an easy calculation implies that
\begin{equation}\label{kappa}
\frac{d}{dt}\int_\gamma |\kappa| ds= -2\sum_{x:\,  \kappa(x,t)=0} |\kappa_s| +\int_\gamma \frac{\kappa}{|\kappa|}g_{ss}  ds\leq (\|\nabla g\|_\infty+\|D^2g\|_\infty) L_t(\gamma).\end{equation}
Recall that, by  Corollary \ref{corper},  $L_t(\gamma)\leq L_0(\gamma) e^{\frac{\|g\|_\infty^2}{ 2 }t }$. 

Following \cite{A:91} we choose a sequence $(x_n,t_n)$ such that 
\begin{itemize}
\item $t_n\in [0, T-\frac{1}{n})$ and $t_n<t_{n+1}$;
\item $k_n=|\kappa(x_n, t_n)|\to +\infty$  and \begin{equation}\label{kn}k_n\sqrt{T-\frac{1}{n}-t_n}=\max_{t\in [0, T-1/n]}\left(\|\kappa\|_{\infty} \sqrt{T-\frac{1}{n}-t}\right)\to +\infty\qquad \text{ as }n\to +\infty.\end{equation}
\end{itemize} 
We define the new parameter $u$ as follows $ u=k_n^2(t-t_n)$, $u\in [-k_n^2 t_n, k_n^2(T-t_n)]$ and the rescaled curve along the sequence
$(x_n,t_n)$ as $\gamma_n(x,u)=k_n(\gamma(x,t(u))-\gamma(x_n,t_n))$, for $x\in [0,1]$. Observe that $\gamma_n(x_n,0)=(0,0)$ and $\kappa_{\gamma_n}(x_n,0)=\kappa_n(x_n,0)=1$.
 Moreover \begin{equation}\label{eqgn} v_n=\frac{d}{du}\gamma_n =\frac{(\kappa +g(\gamma)) \nu}{k_n}=  (\kappa_n+g_n)\nu_n\end{equation} where $g_n(y)= g(y/k_n +\gamma(x_n,t_n))/k_n$. 
 
 Note that for every $\eps>0$, $\omega>0$,there exists $\overline{n}$ such that $\kappa_n^2 \leq 1+\eps$ for $u\in [-k_n^2 t_n, \omega]$. Indeed, using \eqref{kn}, we get 
\[ \kappa_n^2(x,u)= \frac{\kappa^2(x,t(u))}{k_n^2}\leq  \frac{T-1/n-t_n}{T-1/n-t(u)}=\frac{T-1/n-t_n}{T-1/n-t_n-u/k_n^2}.\]
This implies that, on every bounded interval of time, the curvatures of the rescaled curves are uniformly bounded. Moreover, from this, we deduce   uniform bounds  also on  the derivatives of the curvature, using Lemma \ref{stimecurvature} and recalling that $\gamma_n$ satisfies \eqref{eqgn}  and  the fact that $\|\nabla^j g_n\|_\infty=  \|\nabla^j g \|_\infty/ k_n^{j+1}\to 0$.  By the same argument 
of \cite[Theorem 7.3]{A:91}, 
this implies that there exists a subsequence along which the rescaled curves converge smoothly to a smooth, non trivial limit $\gamma_\infty$ defined in $(-\infty, +\infty)$.  Moreover $\gamma_\infty$ evolves by mean curvature flow,  $L_t(\gamma_\infty)=+\infty$ and $\|\kappa_{\infty}\|_\infty=1=|\kappa_{\infty}(0,0)|$.   

We prove now that $\gamma_\infty$ is convex. \\
Let us first observe that by \eqref{kappa} the function 
$t\to \int_{\gamma} |\kappa |ds$ admits a finite limit as $t\to T^-$. Moreover, from \eqref{kappa} we also obtain 
\[
\frac{d}{du}\int_{\gamma_n} |\kappa_n| ds=-2\sum_{\kappa_n(x,u)=0} |(\kappa_n)_{s}|
+\int_{\gamma_n} \frac{\kappa_n}{|\kappa_n|}(g_n)_{s_{n}s_{n}} ds_n.
\] 
So 
\begin{eqnarray*} 
-2\int_{-M}^M\!\! \sum_{\kappa_n=0} |(\kappa_n)_{s}| du &=&  
\int_{\gamma_n} \left(|\kappa_n(x,M)|- |\kappa_n(x,-M)|\right)ds
\\
&& -\int_{-M}^M
\int_{\gamma_n} \frac{\kappa_n}{|\kappa_n|}(g_n)_{s s}dsdu
\\ 
&=& \int_{\gamma } \left(\left|\kappa \left(x,t_n+\frac{M}{k_n^2}\right)\right|
-  \left|\kappa \left(x,t_n-\frac{M}{k_n^2}\right)\right|\right)ds
\\
&& -\int_{-M}^M
\int_{\gamma_n} \frac{\kappa_n}{|\kappa_n|}(g_n)_{s s}dsdu \,.
\end{eqnarray*}
Letting $n\to +\infty$ along the subsequence on which $\gamma_n\to \gamma_\infty$, we get 
\[
\int_{\gamma } \left|\kappa \left(x,t_n+\frac{M}{k_n^2}\right)\right|ds
-\int_{\gamma } \left|\kappa \left(x,t_n-\frac{M}{k_n^2}\right)\right|ds\to 0.
\] 
Recalling  the estimates on the length in Corollary \ref{corper}, we also have
\begin{eqnarray*}
\left| \int_{-M}^M
\int_{\gamma_n} \frac{\kappa_n}{|\kappa_n|}(g_n)_{s s}ds du\right| 
&\le&  
\frac{2}{k_n} \left( \|\nabla g\|_\infty +\frac{\|D^2 g\|}{k_n}\right)L_0(\gamma_n) 
\frac{e^{\frac{ \|g_n\|_\infty M}{2}}- e^{ -\frac{ \|g_n\|_\infty M}{2}}}{\|g_n\|_\infty}
\\
&& \to 0 \qquad \qquad \textrm{as } n\to +\infty\,.
\end{eqnarray*}
In particular, this gives 
\[
-2\int_{-M}^M\sum_{x,  \kappa_n(x,u)=0} |(\kappa_n)_{s}| du \to 0 \qquad \text{as }n\to +\infty,
\]
and we can conclude as in \cite[Theorem 7.7]{A:91} that $\gamma_\infty$ is a convex eternal solution to the curvature flow, 
that is, $\gamma_\infty$ is the so-called {\it Grim Reaper}. 
\end{proof}

\subsection{The embedded case}\label{seciso}

In this section we strengthen Proposition \ref{protime} in the case of 
embedded planar curves.

Following \cite{Hu:98} we define 
\begin{equation}\label{defeta}
\eta(t):=\inf_{x<y}\frac{|\gamma(x,t)-\gamma(y,t)|}{L_{x,y}(t)},
\end{equation}
where 
\[
L_{x,y}(t) = \int_x^y |\gamma_x(\sigma,t)|\,d\sigma.
\]
Notice that the infimum in \eqref{defeta} is in fact a minimum, moreover $\eta$ is a continuous function in $[0, T)$, where $T$ is the first singularity. 
Since the initial curve is embedded, we have $\eta(0)>0$.
Let now 
\[
\E(t):=\left\{(x,y):\,x<y\ {\rm and}\ \eta(t) = \frac{|\gamma(x,t)-\gamma(y,t)|}{L_{x,y}(t)}\right\}.
\]
Notice that, if $\eta(t)<\sqrt 2/2$, we have the estimate 
\[
\int_x^y |\kappa| \ |\gamma_x(\sigma,t)|\,d\sigma\ge c:=\frac\pi 2
\]
for all $(x,y)\in\E(t)$, whence we get 
\begin{eqnarray}\label{apio}
\int_x^y \kappa (\kappa + g)|\gamma_x(\sigma,t)|\,d\sigma &\ge & \frac{1}{L_{x,y}}\left(\int_x^y |\kappa||\gamma_x(\sigma,t)|\,d\sigma\right)^2
-\| g\|_\infty\int_x^y |\kappa||\gamma_x(\sigma,t)|\,d\sigma \nonumber \\ &\ge&  c\left(\frac{c}{L_{x,y}}-\| g\|_\infty\right)\geq 0 
\end{eqnarray}
whenever $L_{x,y}\le c/\| g\|_\infty$. 
Moreover, reasoning as in \cite{Hu:98}, from the minimality condition it follows that 
\begin{eqnarray}\label{sipio}
\left(\kappa(x)\nu(x)-\kappa(y)\nu(y)\right)\cdot (x-y)\ge 0,
\end{eqnarray}
for all $(x,y)\in\E(t)$. When $\eta(t)<\sqrt 2/2$, using \eqref{apio} and \eqref{sipio} we then compute
\begin{eqnarray}\label{etap}
\nonumber
\dot\eta \!\! &=& \!\!\!\min_{(x,y)\in\E}
\frac{1}{L_{x,y}}\left(\left[(\kappa(x)+g(x))\nu(x)-(\kappa(y)+g(y))\nu(y)\right]\cdot\frac{(x-y)}{|x-y|}
+\eta\!\!\int_x^y \!\!\kappa(\kappa+g)|\gamma_x|\,d\sigma\right)
\\
\!\! &\ge& \!\!\!\min_{(x,y)\in\E} -\frac{2\|g\|_\infty}{L_{x,y}}+
\frac{c\eta}{L_{x,y}}\left(\frac{c}{L_{x,y}}-\| g\|_\infty\right)
\\ \nonumber 
\!\! &\ge& \!\!\!\min_{(x,y)\in\E}-\left(2+\frac{\pi\sqrt 2}{4}\right)\frac{\|g\|_\infty}{L_{x,y}}
+ \frac{\pi^2}{4L_{x,y}^2}\,\eta\,.
\end{eqnarray}
\begin{theorem}\label{teoemb}
Let $\gamma_0$ be an embedding and let $T$ be the maximal existence time of \eqref{eqpar}. 
Then 
\begin{equation}\label{eqtime}
T \ge c\left( \gamma_0,\| g\|_\infty\right)\,.
\end{equation}
\end{theorem}

\begin{proof}
Remark \ref{remuno} assures that the statement is true if the 
evolution develops a type I singularity at $t=T$. 

Now we can assume that the evolution develops a type II singularity at $t=T$.
In particular it follows that $\eta(T)=0$. 
Let $\tau := \sup\{ t\in [0,T]:\,\eta(s)>0\ {\rm on\ }[0,t]\}$. Notice that $\tau>0$ 
due to the fact that $\gamma_0$ is an embedding.

The thesis will follow  if we show that $\tau$
is bounded below by a constant depending only on $\gamma_0$ and $\| g\|_\infty$. 
Since $\eta(\tau)=0$, we can find $0\le t_1<t_2\le \tau$ such that 
\[
\eta(t_1)=\bar\eta:=\min\left( \eta(0),\frac{\sqrt 2}{2}\right) 
\qquad \eta(t_2)=\frac{\bar \eta}{2} 
\qquad \eta(t)\in \left( \bar\eta,\frac{\bar\eta}{2}\right) 
\ {\rm for\ all\ }t\in \left(t_1,t_2\right).
\]
In particular, 
letting $a:= (2+\pi\sqrt 2/4)\| g\|_\infty$ and $b:= \pi^2/4$, from \eqref{etap} we have 
\[
\dot\eta \ge -\frac{a}{L_{x,y}} + \frac{b}{L_{x,y}^2}\eta \ge -\frac{a}{L_{x,y}} 
+ \frac{b\,\bar\eta}{2L_{x,y}^2} \ge -\frac{a^2}{2b\bar\eta}\,,
\]
which implies
\begin{equation*}
\tau\ge t_2-t_1\ge \frac{b\bar\eta^2}{a^2} = \frac{2\bar\eta^2}{\left( 1+\frac{4\sqrt 2}{\pi}\right)^2\| g\|^2_\infty}\,.
\qedhere\end{equation*}
\end{proof}

\subsection{The graph case}\label{secgr}
We assume 
now  that the curve can be parametrized as $\gamma( x, t)= ( x,u( x,t))$, 
with $ x\in [0,1]$ and $u\in C^\infty([0,1])$, with the following periodic-type boundary conditions:
\begin{eqnarray*}
u(0,t)-u(0,0) &=& u(1,t)-u(1,0) 
\\
u_x(0,t) &=& u_x(1,t)
\end{eqnarray*}
for all $t\in [0,T]$.

Notice that $\gamma$ is not a closed curve, and can be extended to a periodic infinite curve.
In this parametrization, equation \eqref{eqpar} becomes
\begin{equation}\label{eqgraph}
u_t =  \frac{u_{ x x}}{1+u_{ x}^2} + g( x,u( x)) \sqrt{1+u_{ x}^2}\,.
\end{equation}
We say that $\gamma$ solves \eqref{eqgraph} if $\gamma( x,t)= ( x,u( x,t))$, 
where the function $u$ solves \eqref{eqgraph}. 

Let us recall the following interpolation inequalities~\cite{Ni:59}.

\begin{proposition}\label{pronir}
Let $u\in H^1([0,1])\cap L^p([0,1])$, with $p\in[2,+\infty]$.
We have 
\begin{equation}\label{int0}
    {\Vert u\Vert}_{L^p}
    \leq C_{p}\,
      {\Vert u_x\Vert}_{L^2}^{\frac{p-2}{2p}}
      {\Vert u\Vert}_{L^2}^{\frac{p+2}{2p}} + B_p  \,{\Vert u\Vert}_{L^2}\,.
  \end{equation}
where the constants $C_{p},B_p$ depend only on $p$.
\end{proposition}

The following inequalities can be easily derived from Proposition \ref{pronir} (see \cite{aubin0}).

\begin{proposition}\label{propinter}
Let $z$ be a smooth function defined on $\gamma$, where $\gamma$ solves \eqref{eqgraph}, 
and let $p\in[2,+\infty]$. We have 
\begin{equation}\label{int1}
    {\Vert z\Vert}_{L^p}
    \leq C_{p}\,
      {\Vert z_s\Vert}_{L^2}^{\frac{p-2}{2p}}
      {\Vert z\Vert}_{L^2}^{\frac{p+2}{2p}}+ B_p\,{\Vert z\Vert}_{L^2}\,.
  \end{equation}
where the constants $C_{p}, B_p$ depend on $p$ but are independent of $\gamma$.
\end{proposition}

In particular, choosing $p=4$, \eqref{int1} becomes
\begin{equation}\label{int4}
    \int_\gamma z^4\,ds = {\Vert z\Vert}_{L^4}^4
    \leq C \left(
      {\Vert z_s\Vert}_{L^2}
      {\Vert z\Vert}_{L^2}^{3} +{\Vert z\Vert}_{L^2}^{4}\right)\,.
\end{equation}

\begin{lemma}\label{lemest}
Let $u$ be a smooth solution of \eqref{eqgraph}, and let 
\[
F(x):=\int_0^x \arctan(t)\, dt = x\arctan(x) - \log\sqrt{1+x^2}\,.
\] 
We have
\begin{eqnarray}\label{equno}
\partial_t \int_0^1 \sqrt{1+u_{ x}^2}\,dx &\le& C \int_0^1 \sqrt{1+u_{ x}^2}\,dx
\\ \label{eqdue}
\partial_t \int_0^1 F(u_{ x})\, dx &\le& C \int_0^1 \left(1+u_{ x}^2\right)\,dx
\\
\label{eqtre}
\partial_t \int_0^1 \left(\sqrt{1+u_{ x}^2}\right)^3 \,dx &\le& 
C + C \left( \int_0^1 \left(\sqrt{1+u_{ x}^2}\right)^3\,dx\right)^3, 
\end{eqnarray}
where the constants $C>0$ depend only on $\| g\|_{L^\infty}$.
\end{lemma}

\begin{proof}
Inequality \eqref{equno} can be obtained exactly as \eqref{eqL}.

In order to show \eqref{eqdue}, we compute

\begin{eqnarray}\label{equt}
\nonumber 
\partial_t \int_0^1 F(u_{ x})\,dx &=& \int_0^1  -u_t (\arctan u_{ x})_{ x}\,dx
\\ 
&=& \int_0^1  - u_t^2 + g u_t \sqrt{1+u_{ x}^2}\,dx
\\ \nonumber 
&\le& \int_0^1      \frac{g^2}{4} \left(1+u_{ x}^2\right)\,dx
\end{eqnarray}
which leads to \eqref{eqdue}.

We now prove \eqref{eqtre}.
Letting $e_2=(0,1)\in\R^2$ and $z:= 1/(\tau\cdot e_2) = \sqrt{1+u_{ x}^2}$, from \eqref{eqkg5} we get
\begin{equation}\label{eqz}
z_t = - (\kappa+g)_s z^2 \nu\cdot e_2.
\end{equation}
We compute
\begin{eqnarray*}
\partial_t \int_\gamma z^2 \,ds &=& \int_\gamma 2zz_t - \kappa(\kappa+g)z^2 \,ds
\\
&=& \int_\gamma -2z^3(\kappa+g)_s \nu\cdot e_2 - \kappa(\kappa+g)z^2\,ds
\\
&=& \int_\gamma (\kappa+g)\left( -3\kappa z^2 + 6z^2 z_s \nu\cdot e_2\right)\,ds
\\
&=& 3 \int_\gamma (\kappa+g)z_s\frac{1+2z^2(\nu\cdot e_2)^2}{\nu\cdot e_2}\,ds
\\
&=& 3\int_\gamma - \frac{2z^2-1}{z^2-1} z_s^2 + g z_s \frac{2z^2-1}{\nu\cdot e_2}\,ds
\\
&\le& 3\int_\gamma -\frac{2z^2-1}{z^2-1} z_s^2 + \left(gz\sqrt{2z^2-1}\right)
\left(\sqrt{\frac{2z^2-1}{z^2-1}}\,z_s\right)\,ds
\\
&\le& 3\int_\gamma -z_s^2 + \frac{\|g\|_{L^\infty}^2}{2} z^2 (2z^2 -1)\,ds
\\
&\le& 3\int_\gamma -z_s^2 + \|g\|_{L^\infty}^2 z^4\,ds
\\
&\le& - 3\| z_s\|_{L^2}^2 + C \|g\|_{L^\infty}^2 \left( \| z_s\|_{L^2} \| z\|_{L^2}^{3} + \| z\|_{L^2}^{4}\right)
\\
&\le& C\,\|g\|_{L^\infty}^4 \left( \int_\gamma z^2\,ds\right)^3 + C \|g\|_{L^\infty}^2 \left( \int_\gamma z^2\,ds\right)^2,
\end{eqnarray*}
where we used \eqref{int4} to estimate $\| z\|_{L^4}$. 
\end{proof}

\begin{proposition}\label{progi}
Let $g( x,y)\in C^\infty([0,1]^2)\cap L^\infty([0,1]^2)$, and let $u_0\in C^\infty([0,1])$, with 
$u_0(0)=u_0(1)$. Then, there exists $T>0$ depending only on $\|u_0\|_{W^{1,\infty}}$ and $\|g\|_{L^\infty}$ 
such that equation \eqref{eqgraph} admits a smooth solution $u\in C^\infty([0,1]\times [0,T])$. 

Moreover    $\|u(t, \cdot)\|_{ H^1([0,1]))}\leq K$ for every $t\in[0, T]$, where $K$ depends only on $\|u_0\|_{W^{1,\infty}}$ and $\|g\|_{L^\infty}$.
\end{proposition}

\begin{proof}
By standard parabolic regularity theory \cite{LSU} , it is enough to show that the gradient 
$u_x$ remains bounded for a time $T$ as above.   Theorem \ref{teoemb} gives  
that $\kappa = u_{xx}(1+u_x^2)^{-3/2}\in L^\infty([0,1]\times [0,T])$, for $T $  depending only on $\|u_0\|_{W^{1,\infty}}$ and $\|g\|_{L^\infty}$.
Therefore, by equation \eqref{eqw}, we also get $u_x\in L^\infty([0,1]\times [0,T])$. 

Moreover from   \eqref{eqtre}  we obtain that, eventually choosing a smaller $T$ always depending only on  $\|u_0\|_{W^{1,\infty}}$ and $\|g\|_{L^\infty}$, $u_x\in L^3([0,1]\times [0,T ])$ and then also $u_x\in L^2([0,1]\times [0,T])$,  with norm bounded by a constant  depending only on $T$ and $\|g\|_\infty$. 
\end{proof}

\begin{lemma}\label{lemreg}
We have the continuous embedding 
\[
H^1([0,T],L^2([0,1]))\cap L^\infty([0,T],H^1([0,1]))\hookrightarrow  
C^{\frac 1 2,\frac 1 4}([0,1]\times [0,T]).
\]
\end{lemma}

\begin{proof}
Let $u\in H^1([0,T],L^2([0,1]))\cap L^\infty([0,T],H^1([0,1]))$, and let $(x,t),(y,s)\in [0,1]\times [0,T]$,
with $x<y$, $t<s$. Since $u\in L^\infty([0,T],H^1([0,1]))$, we have 
\begin{equation}\label{ast}
|u(x,t)-u(y,t)|\le \int_x^y |u_x|\,d\sigma\le C \sqrt{x-y}.
\end{equation}
Moreover, being also $u\in H^1([0,T],L^2([0,1]))$, we have
\[
\| u(\cdot,t)-u(\cdot,s)\|^2_{L^2}=
\int_0^1|u(x,t)-u(x,s)|^2\,dx\le (s-t)\int_{[0,1]\times[0,T]}u_t^2\,dxdt \le C(s-t).
\]
By \eqref{int0} with $p=\infty$, this implies 
\begin{eqnarray}\label{eat}
\nonumber 
\| u(\cdot,t)-u(\cdot,s)\|_{L^\infty} \!&\le& \!\!C\left( 
\| u_x(\cdot,t)-u_x(\cdot,s)\|^{\frac 1 2}_{L^2}\| u(\cdot,t)-u(\cdot,s)\|^{\frac 1 2}_{L^2}+
\| u(\cdot,t)-u(\cdot,s)\|_{L^2}\right)
\\ 
\!&\le&
\!\!C\,(s-t)^\frac 1 4.
\end{eqnarray}
The thesis follows from \eqref{ast} and \eqref{eat}.
\end{proof}

\begin{proposition}\label{provo}
Let $u_0\in H^1([0,1])$, with $u_0(0)=u_0(1)$, and 
let $g_n\in C^\infty([0,1]^2)\cap L^\infty([0,1]^2)$, with $\|g_n\|_{L^\infty}\leq C$ for every $n$, 
and let $u_n\in C^\infty([0,1]\times [0,T])$ be the solutions of \eqref{eqgraph} with $g=g_n$, 
given by Proposition \ref{progi}. Then there exists $u\in H^1([0,T],L^2([0,1]))\cap L^\infty([0,T],H^1([0,1]))$
such that $u_n\to u$, up to a subsequence, uniformly on $[0,1]\times [0,T]$.
\end{proposition}

\begin{proof}
By Proposition \ref{progi} there exist $T>0$, depending only on $\|u_0\|_{H^{1}}$ and $C$,
such that the solutions $u_n$ are uniformly bounded in $L^\infty([0,T],H^1([0,1]))$. 
Moreover, using the equality for \eqref{equt}, we obtain
\[\int_0^1 \frac{(u_n)_t^2}{2
} dx\leq     \frac{\|g_n\|_{\infty}^2}{2}\int_0^1    \left(1+ {u_n}_{ x}^2\right)\,dx-\partial_t\int_0^1 F((u_n)_{ x})\,dx\]  
and integrating it in time   
we also get a uniform bound of $u_n$ in $H^1([0,T],L^2([0,1]))$.
It then follows that the sequence $u_n$ converges, up to a subsequence as $n\to  +\infty$, to a limit function 
$u$ in the weak topology of $H^1([0,T],L^2([0,1]))\cap L^\infty([0,T],H^1([0,1]))$.
The uniform convergence follows from Lemma \ref{lemreg}.
\end{proof}

We are interested in studying solutions of \eqref{eqgraph} when $g$ is only a $L^\infty$-function. We consider the simpler case in which $g$
is independent of $u$, i.e. $g(x,y)=g(x)$. In this case we define the following notion of weak solution. 
 
\begin{definition}\label{defw}
We say that a function $u\in H^1([0,T],L^2([0,1]))\cap L^\infty([0,T],H^1([0,1]))$ 
is a weak solution of \eqref{eqgraph} if 
\begin{equation}\label{eqdefw}
\int_{[0,1]\times [0,T]} \left( u_t \varphi 
+ \arctan(u_x)\varphi_x - g ( x )\sqrt{1+u_{ x}^2}\,\varphi\right) \,dxdt = 0
\end{equation}
for all test functions $\varphi \in C^1_c([0,1]\times(0,T))$, with periodic boundary conditions.
\end{definition}

We have the following existence theorem  for weak solution to  \eqref{eqgraph}.
\begin{theorem}\label{teogi}
Let $g( x,y)=g( x)$, with $g\in L^\infty([0,1])$, 
and let $u_0\in W^{2,infty}([0,1])$ with periodic boundary conditions. 
Then, there exists $T>0$ depending only on $u_0$ and $\|g\|_\infty$ 
such that equation \eqref{eqgraph} admits a weak solution 
$u\in W^{1,\infty}([0,T],L^{\infty}([0,1]))\cap L^\infty([0,T],W^{2,\infty}([0,1]))$. 
\end{theorem}

\begin{proof}
Let $g_n\in C^\infty([0,1] )$ be a sequence of smooth functions 
which converge to $g$ weakly* in $L^\infty([0,1] )$. 
By Propositions \ref{progi} and \ref{provo} there exist $T>0$, depending only on $\|u_0\|_{H^{1}}$ and $\|g\|_{L^\infty}$,
and smooth solutions $u_n$ of \eqref{eqgraph} which converge, up to a subsequence, to a limit function 
$u$ in uniformly and in the weak topology of $H^1([0,T],L^2([0,1]))\cap L^\infty([0,T],H^1([0,1]))$.

Let us prove that $u$ is a weak solution of \eqref{eqgraph}. The main point is showing that 
${u_n}_x$ converge to $u_x$ almost everywhere, so that we can pass to the limit in \eqref{eqdefw}.
We compute
\begin{eqnarray}\label{uut}
\nonumber 
\partial_t \frac{u_t^2}{2} &=& u_t\, u_{tt} = u_t\left( \frac{u_{xx}}{1+u_x^2} + g(x)\sqrt{1+u_x^2}\right)_t 
\\
&=& \frac{u_t\,u_{txx}}{1+u_x^2} - 2 \frac{u_x\,u_{xx}}{(1+u_x^2)^2}\left( \frac{u_t^2}{2} \right)_x
+ g \frac{u_x}{\sqrt{1+u_x^2}}\left( \frac{u_t^2}{2} \right)_x
\\
\nonumber 
&\le& \frac{1}{1+u_x^2}\left( \frac{u_t^2}{2} \right)_{xx} 
+\left(g \frac{u_x}{\sqrt{1+u_x^2}}- 2 \frac{u_x\,u_{xx}}{(1+u_x^2)^2}\right)\left( \frac{u_t^2}{2} \right)_x.
\end{eqnarray}
In particular, applying the same computation as \eqref{uut} to $u_n$,  we obtain that $\| {u_n}_t\|_\infty$ is decreasing in time. 
Indeed if $ M_n(t) = \sup_{x\in [0,1]} {u_n}_t^2/2$,  \eqref{uut} gives that $M_n'(t) \le 0$.
Therefore 
$u\in W^{1,\infty}([0,T],L^{\infty}([0,1]))$. Moreover, since $g$ depends only on $ x$ we have
\begin{eqnarray*}
\partial_t \int_0^1 \frac{u_t^2}{2}\,dx &=& \int_0^1u_t\,u_{tt}\,dx = \int_0^1 -\arctan(u_x)_t\,u_{xt} + g\left(\sqrt{1+u_x^2}\right)_tu_t\,dx
\\
&\le& \int_0^1 -\frac{u_{xt}^2}{1+u_x^2} + g u_x  u_t\frac{u_{xt}}{\sqrt{1+u_x^2}}\,dx
\\
&\le&\int_0^1 -\frac{u_{xt}^2}{1+u_x^2}+\frac{ g^2u_x^2  u_t^2}{2}\,dx
\\
&\leq& \frac{1}{2}\|g\|^2_\infty \|u_t\|^2_\infty \|u_x\|_{L^2}^2= C
\end{eqnarray*}
where the constant $C>0$ depends only on $u_0$ and $\|g\|_\infty$. We then get 
\begin{eqnarray}\nonumber 
&&\int_{[0,1]} \left( \arctan(u_x)\right)_x^2 
= \int_{[0,1]} \left(u_t - g\sqrt{1+u_x^2}\right)^2 \,dx
\le C \qquad \forall t\in [0,T]
\\ \label{utx}
&&\int_{[0,1]\times [0,T]} \left( \arctan(u_x)\right)_t^2\,dx\,dt 
=
\int_{[0,1]\times [0,T]} \frac{u_{xt}^2}{1+u_x^2} \,dx\,dt\le C. 
\end{eqnarray}
As a consequence, the function $\arctan({u_n}_x)$ is uniformly bounded in $H^1([0,T],L^2([0,T]))\cap L^\infty([0,T],H^1([0,1]))$. 
Therefore, the sequence $\arctan({u_n}_x)$ converges, up to a subsequence, to $\arctan({u}_x)$ uniformly on $[0,1]\times [0,T]$. 
Since $\arctan$ is injective this implies that the sequence ${u_n}_x$ converges to $u_x$ a.e. on $[0,1]\times [0,T]$,
and we can pass to the limit in \eqref{eqdefw}, obtaining that $u$ is a weak solution of \eqref{eqgraph}.

Finally, being $\arctan({u}_x)$ continuous, possibly reducing $T$ we have that $u_x$ is also continuous 
(hence bounded) on $[0,1]\times [0,T]$.
In particular, recalling \eqref{eqgraph} the uniform bound on $u_t$ implies an analogous bound on $u_{xx}$, that is 
$u\in L^\infty([0,T],W^{2,\infty}([0,1]))$.
\end{proof}

\begin{remark}\label{remgi}\rm
If $u_0$ is only in $H^1([0,1])$, since the sequence $u_n$ is uniformly bounded in $H^1([0,T],L^2([0,T]))$,
reasoning as in Theorem \ref{teogi} we get 
$u\in W^{1,\infty}_{\rm loc}((0,T],L^{\infty}([0,1]))\cap L^\infty_{\rm loc}((0,T],W^{2,\infty}([0,1]))$. 
\end{remark}

We conclude the section with a comparison and uniqueness result for solutions to \eqref{eqgraph}.

\begin{theorem}\label{teocfr}
Let $g( x,y)=g( x)$, with $g\in L^\infty([0,1])$, 
and let $u_1,u_2$
be two solutions to \eqref{eqgraph} such that $u_1(x,0)\le u_2(x,0)$ for all $x\in [0,1]$.
Then
\[
u_1\le u_2 \qquad \textrm{on } [0,1]\times [0,T].
\]
In particular, there is a unique solution to \eqref{eqgraph}, 
given an initial datum $u_0\in W^{2,\infty}([0,1])$.
\end{theorem}

\begin{proof}
Let 
\[
d(t) := \min_{x\in [0,1]} u_2(x,t)-u_1(x,t).
\]
Possibly replacing $u_1(\cdot,0)$ with $u_1(\cdot,0)-\delta$, 
we can assume that $d(0)=\delta>0$. The thesis now follows if we can show that 
$d(t)\ge \delta$ for all $t\in [0,T]$. 

{}From \eqref{utx} it follows 
\[
u_t\in L^2([0,T],H^1([0,1]))\hookrightarrow
L^2([0,T],C^{\alpha}([0,1]))
\]
for all $\alpha<1/2$. Let $w=u_2-u_1$, and choose $t\in [0,T]$ such that $w_t(\cdot,t)\in C^{\alpha}([0,1])$.
Then, for all $x\in [0,1]$ such that $d(t)=w(x,t)$, $w$ is twice differentiable at $x$ 
and we have
\begin{eqnarray}
\nonumber 
w_x &=& (u_2)_{ x}-(u_1)_{ x}= 0\,,
\\ \label{eqww}
w_t &=& \left(
\frac{(u_2)_{ x x}}{1+(u_2)_{ x}^2} - \frac{(u_1)_{ x x}}{1+(u_1)_{ x}^2}
+ g( x) \left( \sqrt{1+(u_2)_{ x}^2}-\sqrt{1+(u_1)_{ x}^2}\right)
\right) 
\\ \nonumber 
&=& \left(
\frac{(u_2)_{ x x}}{1+(u_2)_{ x}^2} - \frac{(u_1)_{ x x}}{1+(u_1)_{ x}^2}
\right)
\\ \nonumber 
&=& \frac{w_{xx}}{ 1+(u_1)_{x}^2} \ge 0\,.
\end{eqnarray}
By \eqref{eqww}, for almost every $t\in [0,T]$ we get
\begin{eqnarray*}
\dot d(t) = \min_{x:\ d(t) = w(x,t)} w_t(x,t) \ge 0
\end{eqnarray*}
which gives the thesis.
\end{proof}

\section{A homogenization problem}\label{secomo}
Given a smooth function $g$ which is periodic on $[0,1]^2$, 
we consider the following homogenization problem 
\begin{equation}\label{eqgomo}  
u_t = \frac{u_{xx}}{1+u_{x}^2} + g\left(\frac{x}{\eps},\frac{u}{\eps}\right) \sqrt{1+u_{x}^2}\,
\end{equation} with initial data $u(x,0)=u_0(x)$.
We point out that existence of traveling wave solutions for \eqref{eqgomo} has been established in \cite{DKH:08},
whereas in \cite{LC:09} (see also \cite{CB:04}) the authors discuss the uniqueness of traveling waves 
and characterize the asymptotic speed in some particular case.

By Proposition \ref{progi} there exists $T>0$ independent of $\eps$ and a family of smooth solution 
$u_\eps$ of \eqref{eqgomo}, which are uniformly bounded in $H^1([0,T],L^2([0,1]))\cap L^\infty([0,T],H^1([0,1]))$.
In particular, as in Propositon \ref{provo}, we can pass to the limit, up to a subsequence as $\eps\to 0$, 
and obtain that $u_\eps$ converge uniformly on $[0,1]\times[0,T]$ 
to a limit function $u\in H^1([0,T],L^2([0,1]))\cap L^\infty([0,T],H^1([0,1]))$.

Assume now $u_0\in W^{1,\infty}([0,1])$, with Lipschitz constant $L>0$.
Due to the comparison principle and the periodicity of $g$, for all $N\in\mathbb N$ we have the estimate
\begin{equation}\label{eqnep}
\vert u_\eps(x,t)-u_\eps(x+N\eps,t)\vert \le ([ L] +1 )N\eps,
\end{equation}
where $[L]$ denotes the integer part of $L$.
Passing to the limit in \eqref{eqnep} as $\eps\to 0$, we get
\[
\vert u(x,t)-u(y,t)\vert \le L |y-x|,
\]
that is the norm $\| u(\cdot,t)\|_{W^{1,\infty}}$ is non increasing in $t$. 
We expect this bound to be still true for the approximating sequence $u_\eps$,
which would imply that we can take $T=+\infty$.

Finally, when $g$ depends only on $x$, by Theorem \ref{teogi} we 
have the following result.

\begin{theorem}\label{teomo}
Let $g( x,y)=g( x)$, with $g\in L^\infty([0,1])$, 
and let $u_0\in W^{2,\infty}([0,1])$ with periodic boundary conditions. 
Then, there exists $T>0$ depending only on $u_0$ and $\| g\|_\infty$ 
such that the solutions $u_\eps$ to \eqref{eqgomo} converge
in $W^{1,\infty}([0,T],L^{\infty}([0,1]))\cap L^\infty([0,T],W^{2,\infty}([0,1]))$, as $\eps\to 0$, 
to the (unique) solution $u$ of
\begin{equation}\label{eqglim}
u_t = \frac{u_{xx}}{1+u_{x}^2} + \left(\int_0^1 g(x)\,dx\right) \sqrt{1+u_{x}^2}\,.
\end{equation}
In particular, $u\in C^\infty([0,1]\times (0,T])$.
\end{theorem}

In the generale case, we can determine the limit equation satisfied by $u$ 
only in a very specific case, that is when $u_0$ is a linear function. 
Indeed, by \cite{DKH:08} (see also \cite{LC:09}) for all $\alpha\in\R$ there exist
global smooth solutions $\hat u_{\alpha,\eps}$ of \eqref{eqgomo}, with average slope $\alpha$, 
which are either  {\em stationary waves}, that is  
\[\hat u_{\alpha,\eps}(x,t)=\hat u_{\alpha,\eps}(x ,0)\qquad \forall \ (x,t)\in \R^2,
\]   
or  {\em pulsating waves}, that is there exist $T>0$ and a vector $(v_1,v_2)\in\mathbb Z^2$,
depending on $(\alpha,\eps)$ and such that 
\[ 
\hat u_{\alpha,\eps}\left(x+ \eps v_1 ,t+ \eps T\right)
=\hat u_{\alpha,\eps}(x ,t)+ \eps v_2
\qquad \forall \ (x,t)\in \R^2.
\]  
We let
\[
c(\alpha,\eps)=\frac{(v_1,v_2)\cdot\nu_\alpha}{T}
\qquad {\rm where}\quad 
\nu_\alpha=\left( -\frac{\alpha}{\sqrt{1+\alpha^2}},\frac{1}{\sqrt{1+\alpha^2}}\right)
\]
be the velocity of the wave in the normal direction $\nu_\alpha$,
and we set $c(\alpha,\eps)= 0$ if $\hat u_{\alpha,\eps}$ is a standing wave.
In  particular, in \cite[Section 4]{DKH:08} it is shown that $\hat u_\eps$ can be represented as
\[
\hat u_{\alpha,\eps}(x,t) = \alpha x + c(\alpha,\eps)\sqrt{1+\alpha^2}\,t + \omega(\eps)
\qquad \forall \ (x,t)\in [0,1]\times [0,T],
\] 
with $\omega(\eps)\to 0$  as $\eps\to 0$.
Integrating \eqref{eqgomo} on $[0,1]$ and reasoning as in \cite[Section 4]{CB:04}, we obtain
\[
\lim_{\eps\to 0} c(\alpha,\eps) = c(\alpha):= 
\left\{\begin{array}{ll}
0 & \textrm{if $G(s)=0$ for some $s\in [0,1]$}
\\
\left( \int_0^1\frac{1}{G(s)}\,ds\right)^{-1}
& \textrm{otherwise,}
\end{array}\right.
\]
where 
\[
G(s) := \lim_{L\to\infty}\frac{1}{L}\int_0^L g(x,\alpha x+s)\,dx.
\]
Notice that $c(\alpha)=\int_{[0,1]^2}g$ for all $\alpha\not\in\mathbb Q$.
Using the comparison principle for solutions to \eqref{eqgomo},
we can use $\hat u_{\alpha,\eps}$ as barriers for the solutions $u_\eps$ 
starting from $u_0=\alpha x$ and obtain that 
\[
\lim_{\eps\to 0} u_\eps(x,t)=u(x,t)=\alpha x + c(\alpha)\sqrt{1+\alpha^2}\,t, 
\]
for all $(x,t)\in [0,1]\times [0,T]$. Notice that the function $\alpha\to c(\alpha)$ 
is in general not continuous.

\subsection{Asymptotic analysis of the limit problem}\label{secomo2}

A more general framework to study the the homogenization problem \eqref{eqgomo} for $g$ depending on both variables    
 would be the level set method and the theory of viscosity solutions. 

We consider   functions $U_\eps:\R^2\times[0, +\infty)$ whose $0$ level sets coincide  with the graphs of the solutions $u_\eps$  to \eqref{eqgomo} by setting $\{U_\eps( x,y ,t)=0\}=\{u_\eps( x,t)=y\}$. These functions satisfy the  associated level set equation in $\R^2\times [0, +\infty)$    \begin{equation}\label{levelset}  
U_t^\eps    = \tr\left[\left(\mathbf{I}-\frac{DU^\eps \otimes DU^\eps }{|DU^\eps |^2}\right)D^2 U^\eps \right]+  g\left(\frac{ x}{\eps}, \frac{y}{\eps}\right)  |DU^\eps | 
\end{equation} 
with initial data $U(0, x,y) =U_0( x,y)$. 

The analysis of the asymptotic behaviour of $U_\eps$ as $\eps \to 0$ using viscosity solutions theory is based essentially  on 
two steps. First of all, we identify the limit or effective equation, solving appropriate ergodic problems (called \emph{cell problems}),  obtained roughly speaking by substituting to the function  $U_\eps$ in the equation \eqref{levelset} an asymptotic formal expansion in $\eps$.  Once that the limit problem has been 
defined, it is important to check that it satisfies a comparison principle for viscosity sub and supersolutions. 
The second step is to show that $U_\eps$
converge uniformly as $\eps\to 0$ to a function $U$, solution to the limit problem. This can be obtained using  
the so-called Barles-Perthame semilimits \cite{BP} and the perturbed test function method introduced in \cite{e}. 

For equation \eqref{levelset} only the first step can be at the moment carried out, since the
effective differential operator that we obtain is actually discontinuous and there is not a satisfactory viscosity theory for such problems.  
The only simple case in which the limit operator is continuous and satisfies comparison principle is under the assumption that $\int_{[0,1]^2} g(x,y)dxdy=0$. Nevertheless in that case, the perturbed test function method does not apply, since the first corrector in the 2-scal expansion of $U_\eps$ (see \eqref{exp}) is discontinuous. 

We consider the following formal expansion of the solution $U_\eps$ to \eqref{levelset}.
\begin{equation}\label{exp} U_\eps( x,y,t)= U( x,y,t)+\eps \chi_1\left(\frac{ x}{\eps}, \frac{y}{\eps}\right)+\eps^2 \chi_2\left(\frac{ x}{\eps}, \frac{y}{\eps}\right). 
\end{equation} 
For every $p\in\R^2$, we define the average of $g$ on the normal spaces to $p$ as follows: 
\[G_p(s)=\lim_{R\to +\infty} \frac{1}{\pi R^2}\int_{|z|\leq R, \ z\cdot p=0} g\left(s\frac{p}{|p|}+z\right)dz, \qquad s\in \R. \]
Note that  $\lim_{L\to +\infty}\frac{1}{L}\int_0^L G_p(s)ds= \int_{[0,1]^2} g(x,y)dxdy$. 

\paragraph{First cell problem.} 
For every $p\in \R^2\setminus\{0\}$,   
we define $ \overline{c}\left({p}/{|p|}\right)$ as the unique  constant such that 
the following cell problem  admits a possibly discontinuous bounded viscosity solution $\chi_1$: 
\begin{equation}\label{cell} G_p \left(y\cdot \frac{p}{|p|}\right) |p+D\chi_1|=  \overline{c}\left(\frac{p}{|p|}\right) |p|, \qquad y\in\R^2.  \end{equation}  
The explicit formula for this constant is 
\begin{equation}\label{heffective} \overline{c}\left(\frac{p}{|p|}\right):= \left\{\begin{array}{lll} 0 & & \text{ if  }  G_p(s)=0 \text{ for  some }s; \\
\\ 
\left(\lim_{L\to +\infty}\frac{1}{L}\int_0^L\frac{ds}{G_p(s)} \right)^{-1} & & \text{ if either } G_p(s)>0 \text{ or }G_p(s)<0 \ \forall s.
\end{array}\right. 
\end{equation} 
Indeed, when $G_p(s)=0$ for some $s\in\R$, if there exists a constant for which the problem has a viscosity solution, then this constant has
to be $0$. Moreover it is easy to show that the periodic, bounded discontinuous solution of the equation $D\chi_1=-p $ in the set $\R\setminus\{s\ | G_p(s)=0\}$ also solves \eqref{cell} in the viscosity sense.
   
\noindent If $G_p(s)\ne 0$ for every $s\in\R$, it is well known that 
there exists a unique constant $\overline{c}(p/|p|)$ for which the cell problem 
\eqref{cell} has a bounded continuous solution. 
The explicit representation of the constant 
can be obtained by integrating \eqref{cell} on $p^\perp$.

\begin{remark}\upshape Observe that $\overline{c}(p)=(\int_{[0,1]^2} g(x,y)dxdy)|p|$ for every $p$ such that $p\cdot q\neq 0$ for every $q\in \Z^2$,  i.e. $\overline{c}(p/|p|)$ is constant almost everywhere on $S^1$.   In particular,
if $\int_{[0,1]^2} g(x,y)dxdy=0$, 
then necessarily $\overline{c}(p) $ is constantly equal to $0$. 
However, in general the map $p\to \overline{c}(p/|p|)$ is not constant and not even continuous.  For example, if we consider the case of $g$ depending only on the first variable $g( x,y)=g( x)$, with $\int_0^1 g(x)dx\neq 0$, then $\overline{c}(p/|p|)=\int_0^1 g(x)dx$ for every $p\neq(0,1)$  and  $\overline{c}(0,1)=1/\int_0^1 g(x)dx$ if $g(x)\neq 0$ for every $x$, $ \overline{c}(0,1)=0$ otherwise. 
\end{remark} 

\paragraph{Second cell problem.}
The second corrector $\chi_2$ is defined as a continuous bounded viscosity solution to the cell problem:   
\[-\tr\left[\left(\mathbf{I}-\frac{p+D\chi_1\otimes  p+D\chi_1}{|p+D\chi_1|^2}\right)D^2\chi_2\right]=
G_p\left(y\cdot \frac{p}{|p|}\right)|p+D\chi_1| -g(y)| p+D\chi_1|  
\qquad y\in\R^2.\]
 
\paragraph{The limit equation.}
If we substitute in \eqref{levelset} the formal expansion \eqref{exp}, recalling the characterization of $\chi_1, \chi_2$ as solutions of appropriate cell problems, we  obtain  that the term $U$ in the expansion formally satisfies the equation  
\begin{equation}\label{levelseteff}  
U_t     = \tr\left[\left(\mathbf{I}-\frac{DU \otimes DU  }{|DU  |^2}\right)D^2 U \right]+  \overline{c}\left(\frac{DU}{|DU|}\right) |DU  | 
\end{equation} 
which is therefore the effective or limit equation for the homogenization problem \eqref{levelset}. 
Note that this equation is again the level set equation of curvature flow with forcing term $\overline{c}(\nu)$, 
which is in general not continuous. 


\end{document}